\documentclass[12pt]{amsart}
\usepackage{latexsym}
\usepackage{amssymb}
\usepackage{amsmath}
\usepackage{epsfig}
\usepackage{ae}
\usepackage{color}

\usepackage{amssymb}
\usepackage{psfrag}
\usepackage{enumerate}
\usepackage[latin1]{inputenc}

\newtheorem{theorem}{Theorem}
\newtheorem{proposition}[theorem]{Proposition}
\newtheorem{lemma}[theorem]{Lemma}
\newtheorem{corollary}[theorem]{Corollary}

\theoremstyle{definition}
\newtheorem{definition}[theorem]{Definition}
\newtheorem{remark}[theorem]{Remark}
\newtheorem{example}[theorem]{Example}

\newcommand{\QQ}{{\mathbb Q}}
\newcommand{\RR}{{\mathbb R}}
\newcommand{\SSS}{{\mathbb S}}

\newcommand{\cU}{{\mathcal U}}
\newcommand{\cL}{{\mathcal L}}
\newcommand{\cF}{{\mathcal F}}

\newcommand{\vmod}{v_\text{mod}}

\begin{document}

\vspace*{7ex}

\title[Regular corank one Poisson manifolds] {Codimension one symplectic foliations and regular Poisson structures}

\author{Victor Guillemin}\address{ Victor Guillemin, Department of Mathematics, Massachussets Institute of Technology, Cambridge MA, US, \it{e-mail: vwg@math.mit.edu}}
\author{Eva Miranda}\address{ Eva Miranda,
Departament de Matemàtica Aplicada I, Universitat Politècnica de Catalunya, Barcelona, Spain, \it{e-mail: eva.miranda@upc.edu}}
            \thanks{Eva Miranda is partly supported by the DGICYT/FEDER project MTM2009-07594: Estructuras Geometricas: Deformaciones, Singularidades y Geometria Integral. Ana Rita Pires {was} partly supported by a grant SFRH / BD / 21657 / 2005 of Fundação para a Ciência e Tecnologia. }
            \author{Ana Rita Pires} \address{Ana Rita Pires, Department of Mathematics, Massachussets Institute of Technology, Cambridge MA, US, \it{e-mail: arita@math.mit.edu}}

\dedicatory{{Dedicated to the memory of Paulette Libermann whose
cosymplectic manifolds play a fundamental role in this paper.}}

\date{\today}

\begin{abstract}
In this short note we give a complete characterization of a certain class of compact corank one Poisson manifolds, those equipped with a closed one-form defining the symplectic foliation and a closed two-form extending the symplectic form on each leaf. If such a manifold has a compact leaf, then all the leaves are compact, and furthermore the manifold is a mapping torus of a compact leaf.

These manifolds and their regular Poisson structures admit an extension as the critical hypersurface of a Poisson $b$-manifold as we will see in \cite{guimipi}.
\end{abstract}
\maketitle


\section{Introduction}

Given a regular Poisson structure we have an associated symplectic foliation $\cF$ given by the distribution of Hamiltonian vector fields. In this short paper we study some properties of codimension one symplectic foliations for regular Poisson manifolds and define some invariants associated to them.

In Section \ref{first invariant} we introduce the first invariant, associated to the defining one-form of the foliation. We will see in Section \ref{vanishing first} that this invariant measures how far a Poisson manifold is from being unimodular. When this invariant vanishes we can choose a defining one-form for the symplectic foliation which is closed. In particular when this invariant vanishes, the Godbillon-Vey class of the foliation vanishes too (as had been previously observed by Weinstein in \cite{Weinstein2}).

In Section \ref{second invariant}, we introduce another invariant also related to the global geometry of these manifolds.
The second invariant measures the obstruction to the existence of a closed two-form on the manifold that restricts to the symplectic structure on each leaf.
 This invariant had been previously studied by Gotay in \cite{gotay} in the setting of coisotropic
 embeddings. This invariant also has a nice interpretation for symplectic fiber bundles (see
 \cite{weinsteinetal}).



In section \ref{vanishing first and second}, we explore some of the global implications of the vanishing of these two invariants for compact manifolds. In particular we show that if they vanish and the foliation $\cF$ has at least one compact leaf
 $L$, then all leaves are compact and $M$ itself is the mapping torus associated with the holonomy map of $L$ onto itself. (In particular, the leaves of $\cF$ are the fibers of a fibration of $M$ over $\SSS^1$.)

In Section \ref{vanishing first and second Poisson}, we give the
Poisson version of this mapping torus theorem. These manifolds turn
out to be {\emph {cosymplectic manifolds} in the sense of Libermann
\cite{libermann}}\footnote{ {A cosymplectic manifold is manifold
$M^{2n+1}$ together with a closed one-form $\eta$ and a closed
two-form, $\omega$ such that $\eta\wedge\omega^n$ is a volume
form.}}.

We also briefly describe  another global consequence of the vanishing of these invariants: A $2n$-dimensional Poisson manifold $(M, \Pi)$, is a Poisson $b$-manifold if the section $\Pi^n$ of $\Lambda^n(M)$ intersects the zero section of this bundle transversally. For such manifolds it is easy to see that this intersection is a regular Poisson manifold with codimension one symplectic leaves and that its first and second invariants vanish.
A much harder result (which will be the topic of a sequel to this paper) is that the converse is also true.

As an application of theorem \ref{poissonmappingtorus} we give an explicit description of the Weinstein's groupoid integrating the Poisson structure of these regular manifolds in the case there is a compact leaf. We do it in section \ref{vanishing first and second algebroid}.

The main result of this paper (Theorem \ref{poissonmappingtorus})
{is related to} previous results of HongJun Li \cite{li} where a
symplectic version of Tischler's theorem \cite{tischler} for
cosymplectic manifolds is given. {Tischler's theorem asserts that a
compact manifold that admits a closed one-form is a fibration over a
circle, but not that this fibration is the one given by the
prescribed closed one-form.} Li's main theorem in \cite{li} asserts
that cosymplectic manifolds are equivalent to symplectic mapping
tori. The main difference with our result is that we work precisely
with the foliation given by the closed one-form (whose existence is
guaranteed by the vanishing of the first invariant) and with the
Poisson structure guaranteed by the vanishing of the second
invariant. The construction provided by Tischler and Li for this
foliation is purely topological  and does not {apply to} the
foliation given by the closed one-form.

{Theorems \ref{poissonmappingtorus} and \ref{further} relate to the
work of M\'{e}lanie Bertelson in \cite{be} because they provide
natural examples of regular foliated manifolds that admit a Poisson
structure. In \cite{be}, Bertelson studies topological obstructions
for a foliated manifold to admit regular Poisson structures and uses
the h-principle in the Poisson context to construct examples of such
structures. Our objects with vanishing first obstruction class may
provide examples of application of Bertelson's results. In
particular, these results ensure that the vanishing of just the
first invariant, together with some additional conditions related to
the defining one-form that it is associated to, implies that there
exists a Poisson structure.}

{\bf Acknowledgements} {We are grateful to the referees of this
paper for their useful comments and for drawing our attention to
many interesting papers in the literature on related matters. The
second author wishes to thank the organizers of the Conference
Poisson 2010 (where part of these results where presented) for their
hospitality and a perfect working atmosphere. Some of the results in
this paper were improved thanks to conversations with Rui Loja
Fernandes and David Martínez Torres. In particular, we are thankful
to David Martínez Torres for useful observations about the
symplectic groupoids integrating these manifolds, for a different
interpretation of the unimodularity condition of Theorem
\ref{thm:unimodular}, and for suggesting Example \ref{Davids
example}, and to Rui Loja Fernandes for pointing out the connection
of the vanishing of our second invariant with the results in his
paper with Marius Crainic \cite{marui2}.}
\section{Introducing two invariants of a foliation}\label{two invariants}

\subsection{The defining one-form of a foliation and the first obstruction class}\label{first invariant}

Let $M$ be a real manifold, $\cF$ a transversally orientable codimension one foliation of $M$, and write $\Omega(M)$ simply as $\Omega$.

\begin{definition}
A form $\alpha\in\Omega^1$ is a \textbf{defining one-form} of the foliation $\cF$ if it is nowhere vanishing and $i_L^*\alpha=0$ for all leaves $L\stackrel{i_L}\hookrightarrow M$.
\end{definition}

A basic property of this defining one-form is:

\begin{lemma}\label{lemma:mu1}
For $\mu\in\Omega^k$, we have $\mu\in\alpha\wedge\Omega^{k-1}$ if and only if $i^*_L\mu=0$ for all $L\in\cF$ .
\end{lemma}


In particular, because $i^*_L(d\alpha)=d(i^*_L\alpha)=0$ for all leaves $L\in\cF$, we have
\begin{equation}\label{equation:dalpha}
d\alpha=\beta\wedge\alpha \text{ for some }\beta\in\Omega^1.
\end{equation}

Thus, $\Omega_1=\alpha\wedge\Omega$ is a subcomplex of $\Omega$ and $\Omega_3=\Omega/\alpha\wedge\Omega$ a quotient complex. Although the spaces $\Omega_1$ and $\Omega_3$ are defined using $\alpha$, they do not depend on the particular choice of the defining one-form. Indeed, let $\alpha$ and $\alpha'$ be distinct defining one-forms for the foliation $\cF$, then we must have $\alpha'=f\alpha$ for some nonvanishing $f\in C^\infty(M)$, thus producing the same complexes $\Omega_1$ and $\Omega_3$. The statement of Lemma \ref{lemma:mu1} simply gives a necessary and sufficient condition for a $k$-form $\mu\in\Omega^k$ to be in $\Omega^k_1$.

Writing $\Omega_2=\Omega$ we have a short exact sequence of complexes
$$0\longrightarrow\Omega_1{\longrightarrow}\Omega_2\stackrel{j}{\longrightarrow}\Omega_3\longrightarrow0.$$

Note that although the form $\beta$ is not unique for a fixed choice of $\alpha$, the projection $j\beta$ is: two distinct forms $\beta$ and $\beta'$ must differ by a multiple of $\alpha$ (by a function). From (\ref{equation:dalpha}) we get $0=d(d\alpha)=d\beta\wedge\alpha-\beta\wedge\beta\wedge\alpha=d\beta\wedge\alpha$,
that is, $\alpha$ and $d\beta$ are dependent, so we must have $d\beta\in\Omega_1$. Then, $d(j\beta)=0$ and we can define the \textbf{first obstruction class} $c_\cF\in H^1(\Omega_3)$ to be
$$c_\cF=\left[j\beta\right].$$

\begin{proposition}
The first obstruction class $c_\cF$ does not depend on the choice of the defining one-form $\alpha$.
\end{proposition}

\begin{proof}
Let $\alpha$ and $\alpha'$ be distinct defining one-forms for the foliation $\cF$, then $\alpha'=f\alpha$ for some nonvanishing $f\in C^\infty(M)$. We have
$$d\alpha'=df\wedge\alpha+f d\alpha=df\wedge\alpha+f\beta\wedge\alpha=(\frac{d f}{f}+\beta)\wedge \alpha',$$
so $\beta'=d(\log \left|f\right|)+\beta$. Thus $[j\beta']=[j\beta]$.
\end{proof}

\begin{theorem}\label{thm: vanishing first obstruction class}
The first obstruction class $c_\cF$ vanishes identically if and only if we can chose $\alpha$ the defining one-form of the foliation $\cF$ to be closed.
\end{theorem}

\begin{proof}
The first obstruction class $c_\cF=[j\beta]$ vanishes identically if and only if $\beta=df+g\,\alpha$ for some $f,g\in C^\infty(M)$. Replacing $\alpha$ by $\alpha'=e^{-f}\alpha$ we get
\begin{align*}
d(e^{-f}\alpha) &=-e^{-f}df\wedge\alpha+e^{-f}d\alpha\\
&=-e^{-f}df\wedge\alpha+e^{-f}\beta\wedge\alpha\\
&=-e^{-f}df\wedge\alpha+e^{-f}df\wedge\alpha+e^{-f}g\,\alpha\wedge\alpha\\
&=0\,.
\end{align*}
\end{proof}

\subsection{The defining two-form of a foliation and the second obstruction class}\label{second invariant}

Assume now that $M$ is endowed with a regular corank one Poisson structure $\Pi$ and that $\cF$ is the corresponding foliation of $M$ by symplectic leaves. Furthermore, assume that the first obstruction class $c_\cF$ vanishes and therefore by Theorem \ref{thm: vanishing first obstruction class} the foliation is defined by a closed one-form $\alpha$. Later, in Theorem \ref{thm:unimodular}, we will show that this happens exactly when the foliation $\cF$ is unimodular.

\begin{definition}
A form $\omega\in\Omega^2$ is a \textbf{defining two-form} of the foliation $\cF$ induced by the Poisson structure $\Pi$ if $i_L^*\omega=\omega_L$ is the
symplectic form on each leaf $L\stackrel{i_L}\hookrightarrow M$.
\end{definition}


Using Lemma \ref{lemma:mu1} again, and because $i^*_L(d\omega)=d(i^*_L\omega)=d\omega_L=0$ for all leaves $L\in\cF$, we have
\begin{equation}\label{equation:domega}
d\omega=\mu\wedge\alpha \text{ for some }\mu\in\Omega^2.
\end{equation}

Again note that although the form $\mu$ is not unique for fixed choices of $\alpha$ and $\omega$, the projection $j\mu$ is. From (\ref{equation:domega}) and $\alpha$ being closed, we conclude that $d\mu\wedge\alpha=0$, that is $\alpha$ and $d\mu$ are dependent, and so $d\mu\in\Omega_1$. Then, $d(j\mu)=0$ and we can define the \textbf{second obstruction class} $\sigma_\cF\in H^2(\Omega_3)$ to be
$$\sigma_\cF=\left[j\mu\right].$$

\begin{proposition}
The second obstruction class $\sigma_\cF$ does not depend on the choice of the defining two-form $\omega$.
\end{proposition}

\begin{proof}
Let $\omega$ and $\omega'$ be distinct defining two-forms for the foliation $\cF$. We must have $\omega'=\omega+\nu$ for some $\nu\in\Omega^2$ such that $i_L^*\nu=0$ for every leaf $L\in\cF$. Then, by Lemma \ref{lemma:mu1}, $\nu=\alpha\wedge\xi$ for some $\xi\in\Omega^1$ and we have:
\begin{align*}
d\omega'&= d\omega+d\nu\\
&=\mu\wedge\alpha-\alpha\wedge d\xi\\
&=(\mu+d\xi)\wedge\alpha.
\end{align*}
Thus, $\mu'=\mu+d\xi$ and $[j\mu']=[j\mu]$.
\end{proof}

\begin{theorem}\label{above}
The second obstruction class $\sigma_\cF$ vanishes identically if and only if we can chose $\omega$ the defining two-form of the foliation $\cF$ to be closed.
\end{theorem}

\begin{proof}
The second obstruction class $\sigma_\cF=[j\mu]$ vanishes identically if and only if $\mu=d\nu+\gamma\wedge\alpha$ for some $\nu,\gamma\in\Omega^1$. In this case, we have $d\omega=d\nu\wedge\alpha$, and so the two-form $\omega'=\omega-\nu\wedge\alpha$ will be closed, with its restriction to an arbitrary leaf L still being $\omega_L$. Conversely, if $\omega$ is closed, then $\mu\wedge\alpha=d\omega=0$, the forms $\mu$ and $\alpha$ are dependent, and $[j\mu]=0$.
\end{proof}

{We will see in Section \ref{section: vanishing} that under certain
conditions, a foliation on $M$ with vanishing first and second
invariants is in fact a fibration, and in that case we can
re-interpret Theorem \ref{above} above in the light of Theorem 1 in
\cite{weinsteinetal}, as characterizing the extension of a closed
two-form on a fiber to a closed two-form on the total manifold.}

\subsection{The modular vector field and modular class of a Poisson manifold}

We follow Weinstein \cite{Weinstein2} for the description of modular vector field and the modular class of a Poisson manifold. A complete presentation of these can also be found in \cite{kosman}.

The modular vector field of a Poisson manifold measures how far Hamiltonian vector fields are from preserving a given volume form. A simple example is that of a symplectic manifold endowed with the volume form that is the top power of its symplectic form: the modular vector field will be zero because this volume form is invariant under the flow of any Hamiltonian vector field.

\begin{definition}
Let $(M,\Pi)$ be a Poisson manifold and $\Theta$ a volume form on it, and denote by $u_f$ the Hamiltonian vector field associated to a smooth function $f$ on $M$. The \textbf{modular vector field} $X_{\Pi}^{\Theta}$ is the derivation given by the mapping
$$f\mapsto \frac{\cL_{u_f}\Theta}{\Theta}.$$
\end{definition}

When both the Poisson structure and the volume form on $M$ are implicit, we denote the modular vector field by $\vmod$.

The modular vector field has the following properties \cite{Weinstein2}:
\begin{enumerate}
\item
The flow of the modular vector field preserves the volume and the Poisson structure:
$$\cL_{X_{\Pi}^{\Theta}}(\Pi)=0, \quad\quad \cL_{X_{\Pi}^{\Theta}}(\Theta)=0;$$

\item
When we change the volume form, the modular vector field changes by:
$$X_{\Pi}^{H\Theta}=X_{\Pi}^{\Theta}- u_{log(H)}$$
where $H\in C^\infty(M)$ is nonvaninshing and positive (we are assuming that our manifold is oriented and
transversally oriented). This implies that the class of the modular
vector field in the first Poisson cohomology group is independent of
the volume form chosen. This class is called the \textbf{modular
class} of the Poisson manifold. A Poisson manifold is called
\textbf{unimodular} if this cohomology class is zero. As remarked
above, symplectic manifolds are unimodular.

\item
If $M$ is 2-dimensional with Poisson structure given by the bracket
$\{x,y\}=f(x,y)$ and volume form $\Theta= dx\wedge dy$, then the
modular vector field is $u_f$, the Hamiltonian vector field
associated to $f$ with respect to the standard
Poisson structure $\Theta^{\sharp}$. In particular, the modular
vector field is tangent to the zero level set of $f$.

\end{enumerate}

\section{The case with vanishing invariants}\label{section: vanishing}

\subsection{Vanishing first invariant: unimodular Poisson manifolds}\label{vanishing first}

Assume again that $(M^{2n+1},\Pi)$ is a corank one Poisson manifold
and $\cF$ is its foliation by symplectic leaves. Fix
$\alpha\in\Omega^1(M)$ and $\omega\in\Omega^2(M)$ defining one- and
two-forms of the foliation, respectively. Let us
compute the modular vector field associated to the volume form
$\Theta=\alpha\wedge\omega^n$:

$$\begin{array}{rcl}
\cL_{u_f}\Theta&=&d(\iota_{u_f}\Theta)\\
&=&d(\iota_{u_f}\alpha\wedge\omega^n)\\
&=&-n\,d(\alpha\wedge df\wedge\omega^{n-1})\\
&=&n\,\alpha\wedge df\wedge\beta\wedge\omega^{n-1},
\end{array}$$
where $\beta\in\Omega^1(M)$ is such that $d\alpha=\alpha\wedge\beta$. On the last step we use the fact that $d\omega\wedge\alpha=0$, from equation \ref{equation:domega}.

By definition of modular vector field, we have
\begin{equation}\label{eq:modular vector field computation 1}
n\,\alpha\wedge df\wedge\beta\wedge\omega^{n-1}=\left(\iota_{v_\text{mod}} df\right) \alpha\wedge\omega^n.
\end{equation}

Furthermore, $v_\text{mod}$ is tangent to the leaves $L$ of $\cF$,\footnote{See for instance section 5 in \cite{Weinstein2}.} so
(\ref{eq:modular vector field computation 1}) implies that
\begin{equation}\label{eq:modular vector field computation 2}
n\,df_L\wedge\beta_L\wedge\omega_L^{n-1}=(\iota_{v_\text{mod}}df_L)\omega_L^n,
\end{equation}
where as before $\omega_L=i_L^*\omega$, and similarly $f_L=i_L^*f$ and $\beta_L=i_L^*\beta$.

On the other hand, because $(df_L\wedge\omega_L^n)$ is a $(2n+1)$-form on a $(2n)$-dimensional manifold $L$, we have
$$\iota_{v_\text{mod}}(df_L\wedge\omega_L^n)=0$$
and hence
\begin{equation}\label{eq:modular vector field computation 3}
(\iota_{v_\text{mod}}df_L)\omega_L^n=df_L\wedge n\,\iota_{v_\text{mod}}\omega_L\wedge\omega_L^{n-1}.
\end{equation}

From (\ref{eq:modular vector field computation 2}) and (\ref{eq:modular vector field computation 3}) we conclude that for all $f_L\in C^\infty(L)$,
$$\omega_L^{n-1}\wedge df_L\wedge(\iota_{v_\text{mod}}\omega_L-\beta_L)=0.$$
This, together with the fact that for a symplectic form $\omega$ and a two-form $\mu$, we have $\omega^{n-1}\wedge\mu=0$ if and only if $\mu=0$, implies the following:

\begin{theorem}\label{thm:weinsteinmod}
Consider the Poisson manifold $(M^{2n+1},\Pi)$ with volume form $\Theta=\alpha\wedge\omega^n$, where $\alpha$ and $\omega$ are defining one- and two-forms of the induced foliation.

Then, the modular vector field is the vector field which on every symplectic leaf $L\in\cF$ satisfies
$$\iota_{v_\text{mod}}\omega_L=\beta_L,$$
where $\beta_L=i^*_L\beta$ is such that $d\alpha=\beta\wedge\alpha$.
\end{theorem}


A corollary of Theorem \ref{thm:weinsteinmod} is the following criterion for unimodularity\footnote{We thank David Martínez-Torres for pointing out that an alternative proof of this can be obtained via formula (4.6) in \cite{kosmannlaurentweinstein}: this alternative proof requires the use of the interpretation of the modular class as one-dimensional representations, in what seems to be an interesting but longer path. See also \cite{evensluweinstein}.} of corank one Poisson manifolds:

\begin{theorem}\label{thm:unimodular}
A corank one Poisson manifold $(M,\Pi)$ with induced symplectic foliation $\cF$ is unimodular if and only if the first obstruction class of the foliation $c_\cF$ vanishes identically.
\end{theorem}

\begin{proof}
Recall that $(M,\Pi)$ is unimodular if and only if there exists a volume form on it such that the corresponding modular vector field is zero. Also, recall from Theorem \ref{thm: vanishing first obstruction class} that $c_\cF$ vanishes identically if and only if there exists a closed defining one-form $\alpha$.

But by Theorem \ref{thm:weinsteinmod}, the modular vector field $\vmod=X^{\alpha\wedge\omega^n}_\Pi$ is zero if and only if $\beta_L=i_L^*\beta=0$ for every leaf $L\in\cF$, which by Lemma \ref{lemma:mu1} is equivalent to $\beta\in\alpha\wedge\Omega$, thus making $\alpha$ closed since $d\alpha=\beta\wedge\alpha.$
\end{proof}

\begin{remark}
Given a transversally orientable foliation with defining one-form
$\alpha$, the Godbillon-Vey class is defined as the class of the
3-form $\beta\wedge d\beta$. From Theorem \ref{thm:unimodular} we
deduce in  particular that for unimodular Poisson manifolds this
3-form vanishes. The converse is not true as observed by Weinstein
in \cite{Weinstein2}.
\end{remark}

\subsection{Vanishing first invariant: a topological result}\label{vanishing first and second}

We begin by recalling Reeb's global stability theorem about codimension one foliations:

\begin{theorem}\label{theorem: reeb}\cite{Reeb}
Let $\cF$ be a transversely orientable codimension one foliation of a compact connected manifold $M$. If $\cF$ contains a compact leaf $L$ with finite fundamental group, then every leaf of $\cF$ is diffeomorphic to $L$.

Furthermore, $M$ is the total space of a fibration $f:M\to\SSS^1$ with fiber $L$, and $\cF$ is the fiber foliation $\{f^{-1}(\theta)|\theta\in\SSS^1\}$.
\end{theorem}




The key point in the proof of Theorem \ref{theorem: reeb} is that such a foliation with a compact leaf has trivial holonomy (see for example \cite{camacho}). Since a foliation defined by a closed one-form has trivial holonomy as well (see again \cite{camacho}, p.80), the proof and conclusions of Reeb's theorem hold in the following case:

\begin{theorem}\label{globalcompact}
Let $\cF$ be a transversely orientable codimension one foliation of a compact connected manifold $M$ with $c_\cF=0$. If $\cF$ contains a compact leaf $L$, then every leaf of $\cF$ is diffeomorphic to $L$.

Furthermore, $M$ is the total space of a fibration $f:M\to\SSS^1$ with fiber $L$, and $\cF$ is the fiber foliation $\{f^{-1}(\theta)|\theta\in\SSS^1\}$.
\end{theorem}

\begin{remark}
A theorem of Tischler \cite{tischler} says that a compact manifold endowed with a non-vanishing closed one-form must be a fibration over a circle. However, this fibration need not coincide with the codimension one foliation defined by the closed one-form. Theorem \ref{globalcompact} asserts that when the foliation contains a compact leaf, then it is itself a fibration over a circle.
\end{remark}

We outline an alternative proof of Theorem \ref{globalcompact} that does not use Reeb's theorem (or rather, does not use the aforementioned key point of its proof):

\begin{proposition}\label{alternative}
Let $\cF$ be a transversely orientable codimension one foliation of
a connected manifold $M$ with $c_\cF=0$. Then:
\begin{enumerate}
\item there exists a nontrivial family of diffeomorphisms $\Phi_t:M\to M$, defined for $t\in(-\varepsilon,\varepsilon)$, that takes leaves to leaves;
\item if $\cF$ contains a compact leaf $L$, then all leaves are compact;
\item and furthermore there exists a saturated neighbourhood $\cU$ of $L$ and a projection $f:\cU\to I$ such that the foliation is diffeomorphic to the foliation given by the fibers of $p$.
\end{enumerate}
\end{proposition}

\begin{proof}
\begin{enumerate}
\item\label{need v} Let $\alpha$ be a closed defining one-form of the foliation $\cF$, and $v$ a vector field on $M$ such that $\alpha(v)=1$; this vector field is transversal to the foliation $\cF$. The flow $\Phi_t$ of the vector field $v$, defined for $t\in(-\varepsilon,\varepsilon)$ for small enough $\varepsilon$, takes leaves to leaves diffeomorphically, since
$$\cL_v\alpha=\iota_v d\alpha + d\iota_v \alpha =0.$$

\item Let $N$ be the union of all compact leaves in $M$. The set $N$ is open: Given a leaf $L$ contained in $N$, the set $$\{\Phi_t(L)|t\in(-\varepsilon,\varepsilon)\}$$
is an open neighborhood of $L$ contained in $N$.

{To see that $M\backslash N$ is open, let $L'$ be a leaf in
$M\backslash N$ and $m\in L'$ a point. There exists a neighborhood
$V$ of $m$ in $L'$ and an $\varepsilon > 0$ such that $\Phi_t(L')$
is well defined for $t\in(-\varepsilon,\varepsilon)$, let
$U=\left\{\Phi_t(L'):t\in(-\varepsilon,\varepsilon)\right\}$ be a
neighborhood of $m$ in $M$. If $m$ is not an interior point of
$M\backslash N$, there exists a compact orbit $L$ which intersects
$U$, i.e., there would exist $m'\in V$ and
$t_0\in(-\varepsilon,\varepsilon)$ such that $\Phi_{t_0}(m')\in L$.
But this implies that $\Phi_{-t_0}(L)$ intersects $L'$, and hence
that $\Phi_{-t_0}(L)\subset L'$, and furthermore that
$\Phi_{-t_0}(L)= L'$, so $L'$ would be compact, which contradicts
the original assumption.}

Since $N$ is nonempty and $M$ is connected, we must have $N=M$.

\item Because $i_L^*\alpha=0$ with $L$ compact, Poincaré lemma guarantees the existence of a tubular neighbourhood $\cU$ of $L$ and function $f$ on it such that the leaf $L$ is the zero level set of $f$ and $\alpha=d f$ on $\cU$. Shrinking $\cU$ as necessary, we can assume it is a saturated neighbourhood and that the leaves are level sets of $f$.
\end{enumerate}
\end{proof}

So far we have proved that foliations with vanishing first invariant $c_\cF$ are locally trivial fibrations in the neighbourhood of a given compact leaf. When the manifold is compact, we can use the transverse vector field $v$ of the proof of (\ref{need v}) in Proposition \ref{alternative} to drag this fiber bundle structure, and by a C\u{e}ch-type construction obtain a global fiber bundle over $\SSS^1$. Furthermore, the choice of a transverse vector field $v$ whose flow takes leaves to leaves gives us an Ehresmann connection, which we use to lift the closed loop on the base of the fibration and thus obtain a holonomy map $\phi:L\to L$, thus obtaining the following:

\begin{corollary}
Let $\cF$ be a transversely orientable codimension one foliation of a compact connected manifold $M$ with $c_\cF=0$, and assume that the foliation contains a compact leaf $L$. Then, the manifold $M$ is the mapping torus\footnote{The mapping torus of $\phi:L\to L$ is the space $\frac{L\times \left[0,1\right]}{ (x,0)\sim (\phi(x),1)}.$} of the diffeomorphism $\phi:L\to L$ given by the holonomy map of the fibration over $\SSS^1$.
\end{corollary}

\begin{remark}
If the foliation $\cF$ satisfies $c_\cF=0$ but has no compact leaves, and instead there exists a leaf $L$ such that $H^1(L, \mathbb R)=0$ (hypothesis of Thurston's theorem \cite{thurston}), then it can be proved that the foliation is a fibration over $\SSS^1$.
\end{remark}

\subsection{Vanishing (first and) second invariant(s): a Poisson result}\label{vanishing first and second Poisson}
We want to look in this section at the case of vanishing first and
second invariants, and for that we must now assume that $M$ is
endowed with a regular corank one Poisson structure and that $\cF$
is the corresponding symplectic foliation. Throughout this section,
we will assume that the manifold $M$ is compact and that there
exists  a compact leaf $L$ (and therefore {by Theorem}
\ref{poissonmappingtorus} all leaves are compact).

\begin{proposition}\label{transversepoisson}
The two invariants $c_\cF$ and $\sigma_\cF$ vanish if and only if there exists a Poisson vector field transversal to the foliation.
\end{proposition}

\begin{proof}
Given $\alpha$ and $\omega$ defining one- and two-forms respectively, let $v$ be the vector field uniquely defined by
\begin{equation}\label{define v}
\alpha(v)=1\text{ and }\iota_v\omega=0.
\end{equation}
Conversely, given a vector field $v$ on $M$ transversal to the foliation $\cF$, these equalities uniquely give us defining one- and two-forms $\alpha$ and $\omega$. We must prove that $v$ is Poisson if and only if only if $d\alpha=d\omega=0$, i.e., $c_\cF=\sigma_\cF=0$.


First, let us assume that $d\alpha=d\omega=0$. Then by (\ref{define
v}), we have $\cL_v\alpha=\cL_v\omega=0$, and we want to show that
$\cL_v\Pi=0$, where $\Pi$ is the Poisson bivector field on $M$.

Because $\alpha$ and $\omega$ are the defining one- and two- forms
of the symplectic foliation induced by $\Pi$ we have,
$$\iota_{\Pi} (\alpha\wedge \omega^n)=n\,\alpha \wedge\omega^{n-1}.$$
Applying $\cL_v$ to both sides and using $\cL_v\alpha=\cL_v\omega=0$ we obtain
$$\iota_{(\cL_v\Pi)}(\alpha\wedge \omega^n)=0,$$
which implies that $\cL_v\Pi=0$ because $\alpha\wedge \omega^n$ is a volume form.

Now let us assume that $v$ is a Poisson vector field, we will first check that $d\alpha=0$. For this we just need to apply $d\alpha$ to any arbitrary pair of vector fields. Since $\alpha$ is the defining one-form of the symplectic foliation, the two-form $d\alpha$ must vanish on pairs of the form $(u_f,u_g)$. For pairs of the form $(v, u_f)$, we have
\begin{equation}\label{eqn:poissonvf}
d\alpha(v,u_f)=v(\alpha(u_f))-u_f(\alpha(v))- \alpha([v,u_f]).
\end{equation}
All three terms on the right vanish, the last one because $v$ being a Poisson vector field implies that $[v,u_f]=u_{v(f)}$. Thus, (\ref{eqn:poissonvf}) becomes $d\alpha(v,u_f)=0$, and so $d\alpha=0$.



To see that $d\omega=0$, it suffices to check it on triples of the form $(v,u_f,u_g)$. Using  $\omega(u_{h_1},u_{h_2})=\{h_1,h_2\}$ and the classical formula
\begin{align*}
d\omega(u,v,w)=u(\omega(v,w))+v(\omega(w,u))+w(\omega(u,v))-\\
-\omega([u,v],w)-\omega([v,w],u)-\omega([w,u],v),
\end{align*}
we obtain
$$d\omega(v,u_f,u_g)=v(\{f,g\})-\{v(f), g\}-\{f,v(g)\},$$
which again vanishes when $v$ is a Poisson vector field.

\end{proof}

As before, the choice of a transverse vector field $v$ gives an Ehresmann connection on the fibration
that we can use to lift the closed loop on the base of the fibration to obtain a holonomy map $\phi:L\to L$.

Because $v$ is a Poisson vector field, its flow drags the symplectic structure of one leaf to define the Poisson structure on $M$ (because $M$ is a symplectic fibration over a one dimensional base, the Guillemin-Lerman-Sternberg condition \cite{gls} guarantees that the minimal coupling structure yields a unique Poisson structure on $M$). Furthermore, the parallel transport of the connection preserves the symplectic structure on the leaves, and in particular the holonomy map $\phi$ is a symplectomorphism:

\begin{theorem}\label{poissonmappingtorus}
Let $M$ be an oriented compact connected regular Poisson manifold of
corank one and $\cF$ its symplectic foliation. If
$c_\cF=\sigma_\cF=0$ and $\cF$ contains a compact leaf $L$, then
every leaf of $\cF$ is symplectomorphic to $L$.

Furthermore, $M$ is the total space of a fibration over $\SSS^1$ and it is the mapping torus of the symplectomorphism $\phi:L\to L$ given by the holonomy map of the fibration over $\SSS^1$.
\end{theorem}

Regular corank one Poisson manifolds with $c_{\cF}=\sigma_{\cF}=0$ are interesting in Poisson geometry because they can be characterized as manifolds which are the critical hypersurfaces of Poisson $b$-manifolds:

\begin{definition}
An oriented Poisson manifold $(M^{2n},\Pi)$ is a \textbf{Poisson $b$-manifold} if the map
$$x\in M\mapsto(\Pi(x))^n\in\Lambda^{2n}(TM)$$
is transverse to the zero section. {In particular, this implies that
the critical set $\{x\in M|(\Pi(x))^n=0\}$ is a hypersurface.}
\end{definition}

\begin{theorem}\label{further}\cite{guimipi}
Let $(M,\Pi)$ be an oriented compact regular Poisson manifold of corank one, $\cF$ its symplectic foliation and $v$ a Poisson vector field transversal to $\cF$.

If $c_\cF=\sigma_\cF=0$, then there exists {a Poisson $b$-manifold $(U,\tilde{\Pi})$, with $M\subset U$ its critical locus, and such that $\tilde{\Pi}$ restricted to $M$ is $\Pi$.}
Furthermore, this extension is unique up to isomorphism among the extensions for which $[v]$ is the image of the modular class under the map
$$H^1_\text{Poisson}(U)\longrightarrow H^1_\text{Poisson}(M).$$
\end{theorem}

Conversely, we will see in \cite{guimipi} that the critical hypersurface of a Poisson $b$-manifold, endowed with the restriction of the Poisson $b$-structure, is always a regular corank one Poisson manifold with vanishing first and second invariants. Note that in Theorem \ref{further} we do not require that $M$ be foliated by compact leaves.

\begin{example}\label{Davids example}
Take  $(N^{2n+1},\pi)$ to be a regular corank one Poisson manifold
with compact leaves, and let $X$ be a Poisson vector field on it.
Consider the product $\SSS^1\times N$ endowed with the bivector
field
$$\Pi=f(\theta)\frac{\partial}{\partial\theta}\wedge X+\pi.$$
This is a Poisson $b$-manifold provided that $f$ vanishes linearly
and the vector field $X$ is transverse to the symplectic leaves of
$N$.

The critical locus of this Poisson $b$-structure has as many copies
of $N$ as $f$ has zeros. Because of Theorem \ref{thm:unimodular},
the fact that $X$ is  a Poisson vector field transverse to the
symplectic foliation implies that the two invariants of $N$ vanish.
 by Theorem \ref{poissonmappingtorus}, each copy of $N$ is a
symplectic mapping torus and the gluing diffeomorphism is the
time-one map of the flow of $X$.
\end{example}

\begin{example}
Let $\mathfrak g$ be the Lie algebra of the affine group in
dimension 2. This is a model for noncommutative Lie algebras in
dimension $2$ and for a basis $e_1$, $e_2$, the bracket is
$[e_1,e_2]=e_2$. From this Lie algebra structure we get the Poisson
bivector field
$$\Pi=y\frac{\partial}{\partial x }\wedge \frac{\partial}{\partial y }.$$
In this example the critical locus is the $x$-axis, which is non-compact, and the its symplectic leaves are the points of that line.
The upper and lower half-planes are open dense symplectic leaves. As
we will see in \cite{guimipi}, this is a prototypical example of
Poisson $b$-manifolds.
\end{example}

\begin{example}
Let $M=\mathbb T^3$ with coordinates $\theta_1,\theta_2,\theta_3$
(we work in a covering) and $\cF$ the codimension
one foliation with leaves given by
$$\theta_3=a\theta_1+b \theta_2+k,\quad k\in\RR,$$
where $a,b, 1\in\RR$ are fixed and independent
over $\QQ$. This implies that each leaf is diffeomorphic to $\RR^2$
\cite{mamaev}. Then,
$$\alpha=\frac{a}{a^2+b^2+1}\,d\theta_1+\frac{b}{a^2+b^2+1}\,d\theta_2-\frac{1}{a^2+b^2+1}\,d\theta_3$$
is a defining one-form for $\cF$ and there is a Poisson structure $\Pi$ on $M$ of which
$$\omega=d\theta_1\wedge d\theta_2+b\, d\theta_1\wedge d\theta_3-a\,d\theta_2\wedge d\theta_3$$
is the defining two-form. Note that $\alpha$ and $\omega$ are
closed, and so the invariants $c_\cF$ and $\sigma_\cF$ vanish. Note
that in this example leaves are non-compact so Theorem
\ref{poissonmappingtorus} does not apply. However, Theorem
\ref{further} still holds, and indeed $(M,\Pi)$ can be extended to
$(U,\tilde{\Pi})$ where $U=M\times(-\varepsilon,\varepsilon)$ and
$\tilde{\Pi}$ is the bivector field dual to the two-form
$$\tilde\omega= d(\log t)\wedge \alpha+ \omega.$$
\end{example}

\subsection{Vanishing (first and) second invariants: explicit integration of the Poisson structure}\label{vanishing first and second algebroid}
A Lie algebroid structure over a manifold $M$ consists of a vector bundle $E$ together with a bundle map $\rho:E\to TM$ and a Lie bracket $[~,~]_E$ on the space of sections satisfying, for all $\alpha, \beta\in \Gamma(E)$ and all $f\in C^{\infty}(M)$, the Leibniz identity:
  $$ [\alpha, f\beta]_E= f[\alpha,\beta]_E+L_{\rho(\alpha)}(f)\beta.$$

Given a Lie groupoid one can naturally associate to it a Lie
algebroid structure, but the converse is not true: given a Lie
algebroid it is not always possible to produce a smooth Lie groupoid
with the prescribed Lie algebroid structure. However, when the Lie
algebroid is in fact a Lie algebra (case of Lie algebroid  over a
point), the integration to a Lie group is guaranteed by Lie's third
theorem. In \cite{marui2}, Marius Crainic and Rui Loja Fernandes
solved the problem of integrability of a Lie algebroid to a Lie
groupoid.


Given a Poisson manifold $(M,\pi)$, there exists a natural Lie algebroid structure on $T^*M$: the Poisson cotangent Lie
algebroid has anchor map $\pi^{\sharp}$ and Lie bracket defined by
 $$ [\alpha, \beta]= L_{\pi^{\sharp}(\alpha)}(\beta)- L_{\pi^{\sharp}(\beta)}(\alpha)- d(\pi(\alpha, \beta)).$$

In this case the integrability problem consists of associating a symplectic groupoid to this Lie algebroid structure, as studied by Marius Crainic and Rui Loja Fernandes in \cite{marui1}. The canonical integration, or Weinstein's groupoid, is a symplectic groupoid integrating the algebroid which is source simply connected.

Recall the following theorem \cite[Corollary 14]{marui1},
\begin{theorem}\label{Crainic-Fernandes}[Crainic-Fernandes]
Let $M$ be a regular Poisson manifold. Then:
\begin{enumerate}
\item If $M$ admits a leafwise symplectic embedding then every leaf of $M$ is a Lie-Dirac submanifold.
\item If every leaf of $M$ is a Lie-Dirac submanifold then $M$ is integrable.
\end{enumerate}
\end{theorem}

The second obstruction class $\sigma_\cF$ can be interpreted via Gotay's embedding theorem \cite{gotay} and measures
 the obstruction for a closed two-form to exist on $Z$ which restricts to $\omega_L$ on each leaf $L$. If $c_\cF=0$, we get a leafwise
  symplectic embedding and Theorem \ref{Crainic-Fernandes} guarantees integrability.
  Furthermore, using Theorem \ref{poissonmappingtorus} we can obtain an explicit characterization of the Weinstein's integrating groupoid:

\begin{corollary}
The Poisson structure on a compact manifold with
vanishing invariants $c_\cF$ and $\sigma_\cF$  and a compact leaf is
integrable (in the Crainic-Fernandes sense) and the Weinstein's
symplectic groupoid is a mapping torus.
\end{corollary}

\begin{proof}
Fix a compact symplectic leaf $L$. Consider the Weinstein's groupoid
of a symplectic leaf $(\Pi_1(L),\Omega)$ and the product
$(\Pi_1(L)\times T^*(\mathbb R),\Omega+d\lambda_{\text{liouville}})$
with the product groupoid structure. Let $f$ be the time-1 flow of
$v$, the Poisson vector field transverse to the symplectic
foliation. {By the construction provided in Theorem
\ref{poissonmappingtorus}, $v$ has periodic orbits. Now let
$\tilde{f}$ be the lift to $\Pi_1(L)$, the Weinstein's groupoid of
$L$ and on this product define the \lq\lq natural lift" of
$\tilde{f}$ , namely the map
$\phi:(\tilde{x},(t,\eta))\mapsto(\tilde{f},(t+1,\eta)).$} This map
preserves the product symplectic groupoid structure and thus by
identification induces a symplectic groupoid structure on the
mapping torus which has $\phi$ as gluing diffeomorphism.
\end{proof}

\end{document}